\documentclass[a4paper,11pt]{article}

\usepackage{amsmath, amsthm, amsfonts, amssymb}
\usepackage[all,cmtip]{xy}
\usepackage{enumerate}

\usepackage{tikz-cd} 

\usepackage{color} 
\usepackage{graphicx}
\usepackage{mathrsfs}

\numberwithin{equation}{section}

\usepackage[nosort,nocompress,noadjust]{cite}

\usepackage[bookmarks=false,hyperfootnotes=false,colorlinks,
    linkcolor={red!60!black},
    citecolor={blue!50!black},
    urlcolor={blue!80!black}]{hyperref}

\renewcommand{\eqref}[1]{\hyperref[#1]{(\ref{#1})}}

\newtheorem{letterthm}{Theorem}

\newtheorem{lettercor}[letterthm]{Corollary}

\newtheorem*{haagerup-theorem}{Relative Bicentralizer Theorem}

\newtheorem{thm}{Theorem}[section]

\newtheorem{lem}[thm]{Lemma}

\newtheorem{cor}[thm]{Corollary}
\newtheorem{prop}[thm]{Proposition}

\theoremstyle{definition}
\newtheorem{rem}[thm]{Remark}

\newtheorem{df}[thm]{Definition}

\newcommand{\R}{\mathbb{R}}
\newcommand{\C}{\mathbb{C}}
\newcommand{\B}{\mathbf{B}}

\newcommand{\Z}{\mathbb{Z}}

\newcommand{\N}{\mathbb{N}}

\newcommand{\cZ}{\mathcal{Z}}
\newcommand{\cU}{\mathcal{U}}

\newcommand{\Ad}{\operatorname{Ad}}

\newcommand{\Aut}{\mathord{\text{\rm Aut}}}

\newcommand{\UCP}{\mathord{\text{\rm UCP}}}

\newcommand{\rL}{\mathord{\text{\rm L}}}
\newcommand{\rB}{\mathord{\text{\rm B}}}

\newcommand{\rb}{\mathord{\text{\rm b}}}

\newcommand{\dom}{\mathord{\text{\rm dom}}}

\newcommand{\rd}{\mathord{\text{\rm d}}}

\newcommand{\supp}{\mathord{\text{\rm supp}}}

\newcommand{\ovt}{\mathbin{\overline{\otimes}}}

\newcommand{\ri}{\text{\rm i}}

\newcommand{\II}{{\rm II}}
\newcommand{\III}{{\rm III}}

\begin{document}

\begin{center}
{\boldmath\LARGE\bf Kadison's problem and ergodicity of the bicentralizer flow }

\bigskip

{\sc by Amine Marrakchi\footnote{UMPA, CNRS ENS de Lyon, Lyon (France). E-mail: amine.marrakchi@ens-lyon.fr}}
\end{center}

\begin{abstract}\noindent
We solve Kadison's problem on the existence of maximal abelian subalgebras in irreducible subfactors $N \subset M$ when $N$ is AFD or when $N$ is with expectation in $M$. In the latter case, we first show that the relative bicentralizer flow of a type $\III_1$ irreducible subfactor with expectation is always ergodic. This also implies that for every irreducible subfactor with expectation $N \subset M$, there exists an AFD subfactor with expectation $P \subset N$ that is irreducible in $M$.
\end{abstract}

\section*{Introduction}
Maximal abelian subalgebras occupy a central position in the theory of von Neumann algebras. The choice of a maximal abelian algebra encodes a particular way of "diagonalizing" the algebra. They serve as "coordinate systems" within the noncommutative environment, allowing one to leverage classical measure-theoretic techniques. This is best illustrated by Feldmann and Moore's theory of cartan subalgebras \cite{FM77}. Another example comes from the theory of unitary group representations : the choice of a maximal abelian subalgebra in the commutant of a representation encodes its decomposition into a measurable field of irreducible representations (see \cite[Section 1.G]{BH19} and the references therein).

In 1967, at the Baton Rouge conference \cite{Ka67}, Kadison asked the following question : \emph{``if $N$ is a subfactor of the factor
$M$ for which $N' \cap M$ consists of scalars, will some maximal abelian *-subalgebra of $N$ be a maximal abelian subalgebra of $M$?"}. A subfactor that satisfies $N' \cap M=\C$ is called \emph{irreducible} or \emph{ergodic} \cite{Po21} and this condition is clearly necessary for $N$ to contain a maximal abelian subalgebra of $M$.

In 1981, Popa solved Kadison's question affirmatively when $M$ is factor of type $\II_1$ \cite{Po81} by constructing the desired maximal abelian subalgebra as a limit of an increasing sequence of finite dimensional abelian subalgebras that are more and more ergodic. This was the starting point of a whole new method pioneered by Popa to construct subalgebras obeying various constraints as an inductive limit of finite dimensional algebras which satisfy more and more of the given conditions. This method was developped and refined over the years \cite{Po83, Po16, Po17, Po21} and led to numerous applications (we refer to Popa's blog \cite{Po25} for an exposition of this results).

In \cite{Po81}, the finite trace assumption on $M$ is crucial as it endows $M$ with a Hilbert space structure that is invariant under conjugation by unitaries in $N$, thus allowing the use of hilbertian averaging arguments. The same arguments still work if we only assume that $N$ is of type $\II$ and that there exists a normal conditional expectation from $M$ onto $N$. This condition is automatically satisfied if $M$ is of type $\II_1$ but it fails for example when $N$ is of type $\II_1$ and $M$ is of type $\II_\infty$. In fact, a counter-example to Kadison's question in the $\II_1 \subset \II_\infty$ case was found by Ge and Popa in \cite{GP98}. This raised the question of finding necessary or sufficient conditions for an ergodic subfactor $N \subset M$ to satisfy Kadison's property. Ge and Popa conjectured that the existence of a normal conditional expectation $E : M \rightarrow N$ might be a sufficient condition, without any restriction on the type of $M$ and $N$ (see the discussion
after \cite[Corollary 5.1]{GP98} and \cite[Problem 7.6]{Po21}).

In \cite{Ma25}, the author confirmed Ge and Popa's intuition by solving Kadison's problem for any subfactor with expectation $N \subset M$ provided that $N$ is not of type $\III_1$ or that $N$ is a factor of type $\III_1$ that satisfies Connes' bicentralizer conjecture \cite{Co85}. In this paper, we complete the solution to Ge and Popa's conjecture by dealing with the case where $N$ is an arbitrary factor of type $\III_1$. We also solve Kadison's question affirmatively in the case where $N$ is \emph{Approximately Finite Dimensional (AFD)} without assuming the existence of a normal conditional expectation from $M$ onto $N$. This is new even when $N$ is of type $\II_1$.

\begin{letterthm} \label{letter kadison}
Let $N \subset M$ be an inclusion of von Neumann algebras with separable preduals. Suppose that one of the following assumption holds :
\begin{enumerate}[(1)]
\item  $N$ is AFD.
\item There exists a faithful normal conditional expectation from $M$ onto $N$.
\end{enumerate}
 Then there exists a maximal abelian subalgebra $A \subset N$ such that $A' \cap M=A \vee (N' \cap M)$. In particular, if $N' \cap M \subset N$ then $A$ is maximal abelian in $M$.
\end{letterthm}

Consider the following properties for a subfactor $N \subset M$ with separable predual (we use the terminology of \cite{Po21}) :
\begin{itemize}
\item MASA-ergodicity (or Kadison's property) : there exists an abelian subalgebra $A \subset N$ scuh that $A' \cap M=A$.
\item $R$-ergodicity : there exists a hyperfinite $\II_1$ subfactor $R \subset N$ such that $R' \cap M=\C$.
\item AFD-ergodicity : there exists an AFD subfactor $P \subset N$ such that $P' \cap M=\C$.
\end{itemize}
In \cite[Theorem 1.2]{Po21}, Popa proved that MASA-ergodicity implies $R$-ergodicity and left the reverse implication as an open question (see \cite{Po25}). Theorem \ref{letter kadison} answers this question affirmatively and shows that the three properties above are all equivalent.

Characterizing the ergodic subfactors $N \subset M$ that satsify the above equivalent properties remains an open question. An obvious necessary condition is the following averaging property called \emph{Dixmier's property} or \emph{MV-ergodicity} in \cite{Po21} :
$$ \forall x \in M, \quad  \overline{\mathrm{conv} \{ uxu^* \mid u \in \cU(N) \} }^{\text{weak*}}  \cap \C \neq \emptyset.$$
However, in \cite[Theorem 1.1]{Po21}, Popa obtained an example of a subfactor that is MV-ergodic but not $R$-ergodic (hence not MASA-ergodic). Therefore MV-ergodicity is necessary but not sufficient. An ergodic subfactor with expectation is automatically MV-ergodic (\cite{Ma19}, see also \cite[Theorem 4.2]{Ma25} for a simpler and cleaner proof) and this fact is the starting point of the proof of Theorem \ref{letter kadison}.(2) as well as \cite[Theorem A]{Ma25} and somehow also in \cite{Po81}.

The proof of  Theorem \ref{letter kadison}.(1) is quite elementary. It uses an inductive construction \emph{"\` a la Popa"}. However, since we do not have a trace as in \cite{Po81}, we use the $\rL^1$-norm as in \cite{Po21} and \cite{Ma25}. The key novelty in our iterative construction is Lemma \ref{finite dimensional to finite dimensional abelian} that replaces the hilbertian averaging argument of \cite[Lemma 2.1] {Po81}. Theorem \ref{letter kadison}.(2) on the other hand is a consequence of Theorem \ref{letter kadison}.(1) combined with the following theorem which establishes AFD-ergodicity of all ergodic subfactors with expectation.

\begin{letterthm} \label{letter thm amenable general}
Let $N \subset M$ be an inclusion of von Neumann algebras with separable preduals and suppose that $N$ is with expectation in $M$. Then there exists an AFD subalgebra with expectation $P \subset N$ such that $P' \cap M=N' \cap M$.
\end{letterthm}

Using Theorem \ref{letter thm amenable general} to reduce Kadison's problem for subfactors with expectation to the AFD case was a key idea of \cite{Ma25}. However, in \cite{Ma25}, we were not able to prove Theorem \ref{letter thm amenable general} in full generality when $N$ is of type $\III_1$. Indeed, when $N$ is of type $\III_1$, \cite[Theorem C]{AHHM18} says that there exists an amenable subfactor with expectation $P \subset N$ such that $P$ is irreducible in $M$ if and only if the so-called \emph{relative bicentralizer flow} $\beta^\varphi : \R^*_+ \curvearrowright \rB(N \subset M,\varphi)$ is ergodic. In \cite{AHHM18}, it was conjectured that this bicentralizer flow is always ergodic. In \cite[Theorem D]{Ma18}, we proved this conjecture when $N=M$. In \cite[Theorem F]{Ma25}, we proved this conjecture for an arbitrary inclusion $N \subset M$ but under the assumption that $N$ itself satisfies Connes' bicentralizer conjecture. In this paper, we get rid of this assumption and prove the ergodicity of the relative bicentralizer flow in full generality. More generally, we determine all the eigenvectors of the bicentralizer flow.

\begin{letterthm} \label{main theorem}
Let $N \subset M$ be an inclusion of von Neumann algebras with expectation, where $N$ is of type $\III_1$. Let $\varphi$ be a faithful normal state on $N$. Then the fixed point algebra of the bicentralizer flow $\beta^\varphi : \R^*_+ \curvearrowright \rB(N \subset M,\varphi)$ is equal to $N' \cap M$.

Moreover, for every $t \in \R$ and every $x \in M$, the following are equivalent :
\begin{enumerate}
\item $x \in \rB(N \subset M,\varphi)$ and $\beta_\lambda^\varphi (x)=\lambda^{\ri t} x$ for all $ \lambda \in \R^*_+$.
\item $xy=\sigma_t^\varphi(y)x$ for all  $y \in N$.
\end{enumerate}
\end{letterthm}

For an inclusion of von Neumann algebras $N \subset M$, we define its \emph{relative $T$-invariant} by
$$ T(N \subset M):=\{ t \in \R \mid \text{there exists a nonzero } x \in M \text{ such that } xy=\sigma_t^\varphi(y)x \text{ for all } y \in N\}$$
where $\varphi$ is any faithful normal semifinite weight on $N$ (the invariant is independent of the choice of $\varphi$ because of Connes's Radon-Nikodym cocycle theorem). It generalizes the classical $T$-invariant of a factor defined by Connes \cite{Co72}.

Thanks to Theorem \ref{main theorem} combined with \cite[Theorem C]{AHHM18}, we obtain the following corollary. 

\begin{lettercor} \label{letter cor amenable III1}
Let $N \subset M$ be an inclusion of von Neumann algebras with expectation and with separable preduals. Suppose that $N$ is a type $\III_1$ factor. Then there exists an amenable subfactor with expectation $P \subset N$ such that $P' \cap M=N' \cap M$.

We can always take $P$ of type $\III_1$. If $\lambda \in (0,1)$, we can take $P$ of type $\III_\lambda$ if and only if there is no nonzero $t \in T(N \subset M)$ such that $\lambda^{\ri t}=1$. This obstruction occurs for at most countably many values of $\lambda$.
\end{lettercor}

In \cite{Ma25}, we proposed a \emph{relative bicentralizer conjecture} that computes explicitely the relative bicentralizer of an arbitrary inclusion of von Neumann algebras with expectation. The computation of the eigenvectors of the bicentralizer flow given by Theorem \ref{main theorem} is exactly the one predicted by this relative bicentralizer conjecture, see Theorem \ref{second part}. 

In the same spirit as Theorem \ref{main theorem}, we can also compute the fixed point algebra and the eigenvectors of the modular flow on the relative bicentralizer $\rB(N \subset M,\varphi)$ for an arbitrary inclusion of von Neumann algebras with expectation $N \subset M$, without assuming that $N$ satisfies the bicentralizer conjecture (see Theorem \ref{eigenvectors modular flow} below and compare to \cite[Theorem 8.13]{Ma25}). As an application, we obtain the following relative version of \cite{MV23} (see \cite[Corollary 8.15]{Ma25} for a partial result). 

\begin{letterthm} \label{ergodic flow inclusion}
Let $N \subset M$ be an irreducible subfactor with expectation and with separable preduals. Then the following are equivalent :
\begin{enumerate}
\item There exists a faithful state $\varphi \in M_*$ such that $\sigma^\varphi : \R \curvearrowright M$ is ergodic and $N$ is globally invariant under $\sigma^\varphi$.
\item The canonical inclusion of continuous cores $c(N) \subset c(M)$ is irreducible.
\end{enumerate}
\end{letterthm}

Let us end this introduction by saying a few words about the proof of Theorem \ref{main theorem}. The starting point of the proof is \emph{Dixmier's property} for the inclusion $N \subset M$. We then construct a binormal state $\Phi \in \B(\rL^2(M))^*$ whose properties encode the weak Dixmier property of the inclusion and then apply the \emph{ultrapower implementation of binormal states} technique of \cite{Ma25}. But in order to apply this technique, one needs to control the behaviour of the state $\Phi$ with respect to the modular operator. The key novelty of this paper is to use the \emph{approximate eigenstate} technique of \cite{Ma18} and then carefully manipulate the state $\Phi$ in order to make it well-behaved with respect to the modular operator, while preserving its binormality.

Altogether, the proof of Theorem \ref{main theorem} is new even in the case where $N=M$, and it is easier than the original proof of \cite[Theorem D]{Ma18}. Indeed, the proof of \cite[Theorem D]{Ma18} relies on the very intricate maximality argument of \cite{Ha85} instead of the ultrapower implementation technique that we use for the proof of Theorem \ref{main theorem}.

\bigskip

{\bf Acknowledgment.} We thank Sorin Popa for his valuable comments.

\tableofcontents

\pagebreak

\section{Kadison's problem for AFD subfactors}

\begin{lem} \label{finite dimensional to finite dimensional abelian}
Let $M$ be a von Neumann algebra and $P \subset M$ be a finite dimensional subalgebra. Then for every $\xi \in M_*$, there exists a finite dimensional abelian subalgebra $A \subset P$ such that
\begin{equation} \label{inequality finite dimensional abelian}
\| \xi |_{A' \cap M} \| - \| \xi|_{P' \cap M} \| \leq  \frac{1}{2} \left( \| \xi \| - \| \xi|_{P' \cap M}    \| \right).
\end{equation} 
\end{lem}
\begin{proof}
Suppose first that $P$ is a factor, i.e.\ $P= M_n(\C)$ for some $n \geq 1$. Let $D$ be the diagonal subalgebra of $P$ and $u$ the cyclic permutation matrix of order $n$. Note that $uDu^*=D$ and that the algebra $P$ is generated by $D$ and $u$. Let $C \subset P$ be the abelian subalgebra generated by $u$. 

There is a unique normal conditional expectation $E : M \rightarrow D' \cap M$ and since $u$ normalizes $D$ and $D' \cap M$ we have $uE(u^*xu)u^*=E(x)$ for all $x \in M$. In particular, if $x \in C' \cap M$ then $E(x) \in C' \cap M$. Therefore, we have 
$$E(C' \cap M ) \subset C' \cap D' \cap M=P' \cap M.$$
This shows that 
$$\| \xi \circ E|_{C' \cap M} \|  \leq \| \xi |_{P' \cap M} \|$$
hence
$$\| \xi |_{C' \cap M} \| \leq \| \xi - \xi \circ E \| + \| \xi \circ E|_{C' \cap M} \| \leq \| \xi -\xi \circ E\| + \| \xi |_{P' \cap M} \|$$
On the other hand, we have 
$$ \| \xi \| = \| \xi - \xi \circ E \| + \| \xi \circ E \| = \| \xi - \xi \circ E \|  + \| \xi |_{D' \cap M} \|.$$
Therefore
$$ \| \xi |_{D' \cap M} \| +  \| \xi |_{C' \cap M} \|  \leq \| \xi \| + \| \xi|_{P' \cap M} \|.$$ 
This inequality implies that either 
$$ \| \xi |_{D' \cap M} \| - \| \xi|_{P' \cap M} \| \leq \frac{1}{2} \left( \| \xi \| - \| \xi|_{P' \cap M} \|    \right) $$
or
$$ \| \xi |_{C' \cap M} \| - \| \xi|_{P' \cap M} \| \leq \frac{1}{2} \left( \| \xi \| - \| \xi|_{P' \cap M}   \| \right)$$
hence we can take $A=D$ or $A=C$ accordingly.

Now, suppose that $P$ is an arbitrary finite dimensional subalgebra of $M$. Let $p$ be a minimal projection in $\cZ(P)$. Then by the first part, for every minimal projection $p$ in $\cZ(P)$, we can find an abelian subagebra $A_p \subset pPp$ such that
\begin{equation} \label{inquality Ap}
\| \xi |_{A_p' \cap pMp} \| - \| \xi|_{(pPp)' \cap pMp} \| \leq  \frac{1}{2} \left( \| \xi|_{pMp} \| - \| \xi|_{(pPp)' \cap pMp}    \| \right).
\end{equation}
Let $A=\bigoplus_p A_p$ where $p$ runs through all  minimal projections of $\cZ(P)$. Then we have (see \cite[Lemma 2.1]{Po81})
$$A' \cap M=\bigoplus_p A_p' \cap pMp,$$
$$ P' \cap M =\bigoplus_p  (pPp)' \cap pMp,$$
$$ \cZ(P)' \cap M =\bigoplus_p  pMp,$$
hence by summing up all the inequalities \ref{inquality Ap}, and using the fact that $\| \xi|_{\cZ(P)' \cap M} \| \leq \| \xi \|$ we obtain the inequality \ref{inequality finite dimensional abelian}.
\end{proof}

\begin{lem} \label{amenable to finite dimensional abelian}
Let $M$ be a von Neumann algebra and let $P \subset M$ be an AFD subalgebra. Then for every $\xi \in M_*$, there exists a finite dimensional abelian subalgebra $A \subset P$ such that
\begin{equation} 
\| \xi |_{A' \cap M} \| - \| \xi|_{P' \cap M} \| \leq  \frac{2}{3} \left( \| \xi \| - \| \xi|_{P' \cap M}    \| \right).
\end{equation} 
\end{lem}
\begin{proof}
Take $\varepsilon > 0$. Since $P$ is AFD, we can find a finite dimensional subalgebra $P_0 \subset P$ such that 
$$ \| \xi |_{P_0' \cap M} \|  \leq \| \xi|_{P' \cap M}    \| + \varepsilon.$$
By Lemma \ref{finite dimensional to finite dimensional abelian}, we can find a finite dimensional abelian subalgebra $A \subset P_0$ such that 
$$ \| \xi |_{A' \cap M} \| - \| \xi|_{P_0' \cap M} \| \leq  \frac{1}{2} \left( \| \xi \| - \| \xi|_{P_0' \cap M}    \| \right).$$
Since 
$$ \| \xi |_{P' \cap M} \|   \leq \| \xi |_{P_0' \cap M} \|  \leq \| \xi|_{P' \cap M}    \| + \varepsilon$$
we conclude that 
$$ \| \xi |_{A' \cap M} \| - \| \xi|_{P' \cap M} \| -  \varepsilon \leq  \frac{1}{2} \left( \| \xi \| - \| \xi|_{P' \cap M}    \| \right)$$
which implies the desired inequality if $\varepsilon$ is small enough.
\end{proof}

\begin{lem}
Let $P \subset M$ be an inclusion of von Neumann algebras where $P$ is AFD. Then for every $\xi \in  M_*$, every $\varepsilon > 0$ and every finite dimensional abelian subalgebra $A_0 \subset P$, there exists a finite dimensional abelian subalgebra $A \subset P$ such that $A_0 \subset A$ and
\begin{equation} 
\| \xi |_{A' \cap M} \| - \| \xi|_{A \vee (P' \cap M)} \| \leq  \varepsilon.
\end{equation} 
\end{lem}
\begin{proof}
By Lemma \ref{amenable to finite dimensional abelian}, for every minimal projection $p \in A_0$, we can find a finite dimensional abelian subalgebra $A_p  \subset pPp$ such that 
\begin{equation} \label{Ap inequality}
\| \xi|_{A_p' \cap pMp} \| - \| \xi|_{(pPp)' \cap pMp } \| \leq \frac{2}{3} \left( \| \xi|_{pMp} \| - \| \xi|_{(pPp)' \cap pMp}  \| \right)
\end{equation} 
Let $A_1 = \bigoplus_{p} A_p$ where $p$ runs through the minimal projections of  $A_0$.
Then $A_0 \subset A_1$ and we have $$A_1' \cap M=\bigoplus_{p} A_p' \cap pMp,$$
$$ A_0 \vee (P' \cap M) =\bigoplus_p  (pPp)' \cap pMp,$$
$$ A_0' \cap M =\bigoplus_p  p Mp.$$
Thus if we sum up the inequalities \ref{Ap inequality} over all minimal projections $p$, we obtain
$$\| \xi |_{A_1' \cap M} \| - \| \xi|_{A_0 \vee (P' \cap M)} \| \leq  \frac{2}{3} \left( \| \xi |_{A_0' \cap M}\| - \| \xi|_{A_0 \vee (P' \cap M)} \|    \right).$$
and since
$$ \| \xi|_{A_1 \vee (P' \cap M)} \| \geq \| \xi|_{A_0 \vee (P' \cap M)} \|, $$  we get 
$$\| \xi |_{A_1' \cap M} \| - \| \xi|_{A_1 \vee (P' \cap M)} \| \leq  \frac{2}{3} \left( \| \xi |_{A_0' \cap M}\| - \| \xi|_{A_0 \vee (P' \cap M)} \|    \right).$$
By iterating this construction, we obtain an increasing sequence $(A_n)_{n \in \N}$ of finite dimensional abelian subalgebras of $P$ such that for every $n \in \N$ we have
$$\| \xi |_{A_{n+1}' \cap M} \| - \| \xi|_{A_{n+1} \vee (P' \cap M)} \| \leq  \frac{2}{3} \left( \| \xi |_{A_n' \cap M}\| - \| \xi|_{A_n \vee (P' \cap M)} \|    \right)$$
hence 
$$\| \xi |_{A_{n}' \cap M} \| - \| \xi|_{A_{n} \vee (P' \cap M)} \| \leq  \left( \frac{2}{3} \right)^n \left( \| \xi |_{A_0' \cap M}\| - \| \xi|_{A_0 \vee (P' \cap M)} \|    \right).$$
The conclusion follows by taking $A=A_n$ for $n$ large enough.
\end{proof}

\begin{thm} \label{kadison AFD}
Let $N \subset M$ be an inclusion of von Neumann algebras with separable preduals. Suppose that $N$ is AFD. There exists a maximal abelian subalgebra $A \subset N$ such that $A' \cap M=A \vee (N' \cap M)$. In particular, if $N' \cap M \subset N$ then $A$ is maximal abelian in $M$.
\end{thm}
\begin{proof}
Let $(\xi_n)_{n \in \N}$ be a dense sequence in $M_*$. We can construct inductively an increasing sequence $(A_n)_{n \in \N}$ of finite dimensional abelian subalgebras of $P$ such that
$$ \| \xi_n|_{A_n' \cap M} \| - \| \xi_n|_{A_{n} \vee (P' \cap M)} \| \leq 2^{-n}.$$
Let $A=\bigvee_{n \in \N} A_n$. Then for every $n \geq 1$, we have
$$ \| \xi_n|_{A' \cap M} \| - \| \xi_n|_{A \vee (P' \cap M)} \| \leq \| \xi_n|_{A_n' \cap M} \| - \| \xi_n|_{A_{n} \vee (P' \cap M)} \| \leq  2^{-n}.$$
Since $(\xi_n)_{n \in \N}$ is dense in $M_*$, it follows that $$\| \xi|_{A' \cap M} \|  = \| \xi|_{A \vee (P' \cap M)} \|$$ for all $\xi \in M_*$. By the Hahn-Banach theorem, we conclude that $A' \cap M = A \vee (P' \cap M)$. But $A$ is not necessarily maximal abelian. Take $B$ any maximal abelian subalgebra of $P$ that contains $A$. Then we have 
$$ B' \cap M \subset A' \cap M=A \vee (P' \cap M) \subset B \vee (P' \cap M)$$
and since the reverse inclusion is obvious, we conclude that $B' \cap M =B \vee (P' \cap M)$.
\end{proof}

\section{Ergodicity of the relative bicentralizer flow}
Let $N \subset M$ be an inclusion of von Neumann algebras with a faithful normal conditional expectation $E_N : M \rightarrow N$. We suppose that $N$ is of type $\III_1$ (not necessarily a factor). We choose a state $\varphi \in M_*$ such that $\varphi=\varphi \circ E_N$. For every subalgebra $A \subset M$ that is globally invariant under $\sigma^\varphi$ we denote by $E_A : M \rightarrow A$ the unique $\varphi$-preseving conditional expectation.

Let $\lambda,\rho : M \rightarrow \B(\rL^2(M))$ denote the left and right actions of $M$ in its standard form \cite{Ha75}. We also use the bimodule notation $\lambda(a)\rho(b)\xi = a \xi b$ for all $\xi \in \rL^2(M)$ and $a,b \in M$. We denote by $\varphi^{1/2} \in \rL^2(M)$ the unique vector in the positive cone of $\rL^2(M)$ that satisfies $\langle a \varphi^{1/2},\varphi^{1/2} \rangle = \langle \varphi^{1/2}a,\varphi^{1/2} \rangle=\varphi(a) $
for all $a \in M$. We let $e_N \in \B(\rL^2(M))$ denote the Jones projection associated to $E_N$, that is the projection on $\overline{N \varphi^{1/2}}=\overline{\varphi^{1/2} N} \cong \rL^2(N)$.

We refer to \cite{AH12} for the theory of ultraproducts of von Neumann algebras, and we refer to \cite{AHHM18} and \cite{Ma25} for the definition of the relative bicentralizer, the relative bicentralizer flow and their main properties.

The following lemma is an application of the ultrapower implementation of binormal states technique developped in \cite[Section 2]{Ma25}. Thanks to this lemma, one can prove powerful results on the bicentralizer by constructing states on $\B(\rL^2(M))$.
\begin{lem} \label{binormal state bicentralizer}
Suppose that $\Phi \in \B(\rL^2(M))^*$ is a state that satisfies 
\begin{enumerate}[(a)]
\item  $\Phi \circ \lambda = \Phi \circ \rho=\varphi$.
\item $\Phi(e_N)=1$. 
\item $\Phi(f(\log\Delta_\varphi))=f(0)$ for every $f \in C_0(\R)$.
\end{enumerate}
Then we have $$\Phi(\lambda(a)\rho(b))=\langle a \varphi^{1/2} b , \varphi^{1/2} \rangle$$ for all $a,b \in \rB(N \subset M, \varphi)$.
\end{lem}
\begin{proof}
The proof is similar to \cite[Theorem 7.17]{Ma25}. By \cite[Theorem 2.9]{Ma25}, we can find some abelian von Neumann algebra $A$, a cofinal ultrafilter $\omega$ on some directed set $I$ and a vector $\xi \in \rL^2((A \ovt M)^\omega)$ such that
$$ \langle (1 \otimes T)^\omega \xi,\xi \rangle = \Phi(T)$$
for all $T \in \B(\rL^2(M))$. In fact, if we choose some faithful normal state $\mu$ on $A$ and let $\phi=\mu \otimes \varphi$, we also have $\xi \in \rL^2((A \ovt M)^\omega_{\phi^{\omega}})$. Since $(1 \otimes e_N)^\omega$ projects $\rL^2((A \ovt M)^\omega)$ onto $\rL^2((A \ovt N)^\omega)$, the condition $\Phi(e_N)=1$ implies that $\xi=(1 \otimes e_N)^\omega \xi \in \rL^2((A \ovt N)^\omega)$. Thus, we have $\xi \in \rL^2((A \ovt N)^\omega_{\phi^{\omega}})$. 

Take $a \in\rB(N \subset M, \varphi)$ such that $a$ is $\sigma^{\varphi}$-analytic. Then $a \varphi^{1/2}=\varphi^{1/2}a'$ where $a'=\sigma_{\ri/2}(a) \in\rB(N \subset M, \varphi)$.

Since $a \in\rB(N \subset M, \varphi)$, we have $1 \otimes a \in \rB(A \ovt N \subset A \ovt M, \phi)$ by \cite[Proposition 7.8]{Ma25}. Moreover, we have $(1 \otimes a) \phi^{1/2} =\phi^{1/2} (1 \otimes a')$. By \cite[Lemma 7.16]{Ma25} (applied to the inclusion $(A \ovt M)^\omega_{\phi^{\omega}} \subset (A \ovt M)^\omega$), we obtain $(1 \otimes a)\xi=\xi (1 \otimes a')$.

Then for every $b \in M$, we get
$$ \Phi(\lambda(a)\rho(b))=\langle (1 \otimes a) \xi (1 \otimes b) , \xi \rangle = \langle \xi (1 \otimes a'b), \xi \rangle = \Phi(\rho(a'b))=\varphi(a'b)=\langle a \varphi^{1/2} b , \varphi^{1/2} \rangle.$$
This shows the desired property when $a $ is $\sigma^\varphi$-analytic. Since $\Phi$ is normal on $\lambda(M)$, we get the desired conclusion for an arbitrary $a \in M$ by approximating it with $\sigma^\varphi$-analytic elements.
\end{proof}

Now, in order to prove Theorem \ref{main theorem}, we only need to construct a state $\Phi \in \B(\rL^2(M))^*$ that satisfies the assumption of Lemma \ref{binormal state bicentralizer} and such that
$$\Phi(\lambda(a)\rho(b))=\langle E_{N' \cap M}(a) \varphi^{1/2} b, \varphi^{1/2} \rangle$$ for all $a,b \in \rB(N \subset M, \varphi)^{\beta^\varphi}$. We will construct it step by step by using the following lemma. 

\begin{lem} \label{ucp maps}
Let $\UCP(\B(\rL^2(M))$ denote the set of all unital completely positive maps  (not necessarily normal) on $\B(\rL^2(M))$.
\begin{enumerate}[(1)]
\item There exists $\zeta \in \UCP(\B(\rL^2(M))$ such that
 \begin{enumerate}
 \item  $\zeta(e_N)=e_N$  
  \item $\zeta \circ \lambda=\lambda  \circ E_{\rB(N \subset M,\varphi)}$.
 \item $\zeta \circ \rho =\rho$.
 \item $\zeta(f(\log \Delta_\varphi))=f(\log \Delta_\varphi)$ for every $f \in C_0(\R)$.
\end{enumerate}

\item For every $t \in \R$, there exists $\theta_t \in \UCP(\B(\rL^2(M))$ such that

 \begin{enumerate}
 \item  $\theta_t(e_N)=e_N$  
  \item $\theta_t \circ \lambda=\lambda \circ \beta_{e^t}^\varphi \circ E_{\rB(N \subset M,\varphi)}$.
 \item $\theta_t \circ \rho =\rho$.
 \item $\theta_t(f(\log \Delta_\varphi))=f(\log \Delta_\varphi - t)$ for every $f \in C_0(\R)$.
\end{enumerate}

\item There exists $\pi \in \UCP(\B(\rL^2(M))$ such that
 \begin{enumerate}
 \item  $\pi(e_N)=e_N$  
  \item $\pi \circ \lambda=\lambda  \circ E_{\rB(N \subset M,\varphi)^{\beta^\varphi}}$.
 \item $\pi \circ \rho =\rho \circ E_{\rB(N \subset M,\varphi)^{\beta^\varphi}}$.
 \item $\pi(f(\log \Delta_\varphi))=f(\log \Delta_\varphi)$ for every $f \in C_0(\R)$.
\end{enumerate}
\end{enumerate}

\end{lem}
\begin{proof}
\begin{enumerate}[(1)]
\item There exists a net $(\varepsilon_i)_{i \in I}$ of positive real numbers converging to $0$ and a net $(\mu_i)_{i \in I}$ of finitely supported probability measures on the unitary group $\cU(N)$ with $$\supp(\mu_i) \subset \{ u \in \cU(N) \mid \| u \varphi - \varphi u \| \leq \varepsilon_i \}$$ such that the net $$\Ad(\mu_i)=\int_{\cU(N)} uau^* \: \rd\mu_i(u)$$ converges strongly to $E_{\rB(N \subset M,\varphi)}(a)$ for every $a \in M$. Let $\zeta \in \UCP(\B(\rL^2(M))$ be an accumulation point of the net of ucp maps $$\Ad(\lambda(\mu_i)) = \int_{\cU(N)} \Ad(\lambda(u)) \: \rd\mu_i(u).$$
Properties $(a),(b)$ and $(c)$ are easy to check. It is sufficient to prove $(d)$ for a function $f \in C_0(\R)$ of the form $f=\widehat{h}$ for some $h \in \rL^1(\R)$. Then
$$f(\log \Delta_\varphi)=\int_\R \Delta_\varphi^{\ri t} \: h(t) \rd t$$
hence $$\Ad(\lambda(\mu_i)) (f(\log \Delta_\varphi))= \int_\R \int_{\cU(N)} \lambda(u) \Delta_\varphi^{\ri t} \lambda(u)^*  \: \rd\mu_i(u) h(t)  \rd t  $$
$$= \int_\R  \lambda\left( \int_{\cU(N)}  u\sigma_t^\varphi(u)^* \: \rd\mu_i(u)  \right) \Delta_\varphi^{\ri t} \: h(t)  \rd t.$$
Now, observe that the net of functions
$$ t\mapsto \int_{\cU(N)} u \sigma_t^\varphi(u)^*  \: \rd\mu_i(u)$$
converge to $1$ in the strong topology and uniformly on all compact subsets of $\R$, thanks to \cite[Lemma 2.7]{Co74}. We conclude that $\Ad(\lambda(\mu_i)) (f(\log \Delta_\varphi))$ converges strongly to $f(\log \Delta_\varphi)$, hence $\zeta(f(\log \Delta_\varphi))=f(\log \Delta_\varphi)$.

\item Take $\zeta \in \UCP(\B(\rL^2(M)))$ as in (1). Take $t \in \R$. Take a free ultrafilter $\omega$ on $\N$. We can find a finite set of partial isometries $F \subset N^\omega$ such that $\sum_{v \in F} vv^*=1$ and $v\varphi^\omega=e^t \varphi^\omega v$ for all $v \in F$. Then $\sum_{v \in F} vxv^*=\beta_{e^t}^\varphi(x)$ for all $x \in \rB(N \subset M,\varphi)$ and $\sigma_s^{\varphi^\omega}(v)=e^{\ri st} v$ for all $s \in \R$.

Take $\zeta$ as in the first item. Write $v=(v_n)^\omega$ for all $v \in F$ and define $\theta_t \in \UCP(\B(\rL^2(M))$ by
$$ \theta_t(T) =\lim_{n \to \omega} \sum_{v \in F} \lambda(v_n) \zeta(T) \lambda(v_n^*).$$
Properties (a), (b) and (c) are clear. For property (d), take a function $f \in C_0(\R)$ of the form $f=\widehat{h}$ for some $h \in \rL^1(\R)$. Then
$$ \theta_t(f(\log \Delta_\varphi)) = \lim_{n \to \omega} \sum_{v \in F} \lambda(v_n) f(\log \Delta_\varphi))  \lambda(v_n^*)$$
$$= \lim_{n \to \omega} \sum_{v \in F} \lambda(v_n) \left( \int_\R \Delta_\varphi^{\ri s} h(s) \rd s \right) \lambda(v_n^*).$$
$$= \lim_{n \to \omega} \sum_{v \in F} \int_\R \lambda(v_n)\lambda(\sigma_{s}^\varphi(v_n)^*) \Delta_\varphi^{\ri s} h(s) \rd s.$$
$$= \int_\R e^{-\ri st} \Delta_\varphi^{\ri s} h(s) \rd s=f(\log \Delta_\varphi -t)$$
because the functions $s \mapsto \sum_{v \in F} v_n\sigma_s^\varphi(v_n)^*$ converge strongly to $s \mapsto e^{- \ri st}$ and uniformly on all compact subsets when $n \to \omega$.

\item For every $t \in \R$, take $\theta_t$ as in (2) and define $\theta'_t \in \UCP(\B(\rL^2(M))$ by  $$\theta'_t(T)= J\theta_t(JTJ)J, \quad T \in \B(\rL^2(M)).$$ Then we have
 \begin{enumerate}
 \item  $\theta'_t(e_N)=e_N$  
  \item $\theta'_t \circ \lambda =\lambda$.
  \item $\theta'_t \circ \rho=\rho \circ \beta_{e^t}^\varphi \circ E_{\rB(N \subset M,\varphi)}$.

 \item $\theta'_t(f(\log \Delta_\varphi))=f(\log \Delta_\varphi + t)$ for every $f \in C_0(\R)$, because $J \Delta_\varphi J=\Delta_\varphi^{-1}$.
\end{enumerate} 
Let $\Theta_t = \theta_t \circ \theta'_t \in \UCP(\B(\rL^2(M))$. Then we have
 \begin{enumerate}
 \item  $\Theta_t(e_N)=e_N$  
  \item $\Theta_t \circ \lambda=\lambda \circ \beta_{e^t}^\varphi \circ E_{\rB(N \subset M,\varphi)}$.
\item $\Theta_t \circ \rho=\rho \circ \beta_{e^t}^\varphi \circ E_{\rB(N \subset M,\varphi)}$.
 \item $\Theta_t(f(\log \Delta_\varphi))=f(\log \Delta_\varphi)$ for every $f \in C_0(\R)$.
\end{enumerate}

Let $S < \R$ be a dense countable subgroup and let $(S_n)_{n \in \N}$ be a F\o lner sequence in $S$.

For every $n \in \N$, define 
$$\pi_n = \frac{1}{|S_n|} \sum_{t \in S_n} \Theta_{t}  \in  \UCP(\B(\rL^2(M))$$
and let $\pi \in \UCP(\B(\rL^2(M))$ be an accumulation point of the sequence $(\pi_n)_{n \in \N}$. Then $\pi$ clearly satisfies property (a) and (d). For properties (b) and (c), observe that since $\beta^\varphi$ is state preserving, we have the weak*-convergence 
$$ \frac{1}{|S_n|} \sum_{t \in S_n} \beta^\varphi_{e^t}(x) \rightarrow  E_{\rB(N \subset M,\varphi)^{\beta^\varphi}}(x)$$
for every $x \in  \rB(N \subset M,\varphi)$.

\end{enumerate}
\end{proof}

Finally, we will need the following lemma that uses the approximate eigenstate technique from \cite{Ma18}.
 \begin{lem} \label{centralizing}
Let $\sigma : \R \rightarrow \Aut( \B(H))$ be the one-parameter flow generated by some unbounded self-adjoint operator $X$. Let $A \subset \B(H)$ be a C*-algebra and suppose that for every $t \in \R$, there exists a map $\theta_t \in \UCP(\B(H))$ such that $\theta_t(a)=a$ for all $a \in A$ and $\theta_t(f(X))=f( X -t)$ for every $f \in C_0(\R)$.

Then for every strongly $\sigma$-invariant state $\Phi \in \B(H)^*$, there exists a state $\widetilde{\Phi} \in \B(H)^*$ such that $\widetilde{\Phi}(a)=\Phi(a)$ for all $a \in A$ and $\widetilde{\Phi}(f(X))=f(0)$ for every $f \in C_0(\R)$.
 \end{lem}
 \begin{proof}
 The set of states $\Phi \in \B(H)^*$ that satisfy the conclusion of the lemma is a weak*-closed convex set.  Therefore, by \cite[Theorem 3.2]{Ma18}, it suffices to prove the lemma when $\Phi$ is an approximate eigenstate. Write $\Phi = \lim_{i \in I} \langle \cdot \xi_i ,\xi_i \rangle$ with $\xi_i \in \dom(X)$ and $\lim_{i \in I} \| (X-t_i) \xi_i \|=0$ for some $t_i \in \R$. Define $\widetilde{\Phi}$ to be a weak*-acuumulation point of $\langle \theta_{t_i}( \cdot ) \xi_i, \xi_i \rangle$. Clearly, we have $\widetilde{\Phi}(a)=\Phi(a)$ for all $a \in A$. Take $f \in C_0(\R)$ such that $f(0)=0$. Then $\lim_i \| f(X-t_i) \xi_i \|=0$. This means that $\widetilde{\Phi}(|f(X)|^2)=0$, hence $\widetilde{\Phi}(f(X))=0$. Since this holds for every $f \in C_0(\R)$ such that $f(0)=0$ and it follows easily that $\widetilde{\Phi}(f(X))=f(0)$ for every $f \in C_0(\R)$.
 \end{proof}

We are now ready to prove the first part of Theorem \ref{main theorem}.

\begin{thm} \label{first part}
Let $N \subset M$ be an inclusion of von Neumann algebras with expectation, where $N$ is of type $\III_1$. Let $\varphi$ be a faithful normal state on $N$. Then the fixed point algebra of the bicentralizer flow $\beta^\varphi : \R^*_+ \curvearrowright \rB(N \subset M,\varphi)$ is equal to $N' \cap M$.
\end{thm}
\begin{proof}
Consider the set $K$ of all states $\Phi \in \B(\rL^2(M))^*$ such that :
\begin{itemize}
\item $\Phi(e_N)=1$
\item $\Phi(\lambda(a) \rho(b) )=\langle E_{N' \cap M}(a) \varphi^{1/2} b, \varphi^{1/2} \rangle$ for all $a,b \in M$
\end{itemize}

Clearly, $K$ is a weak*-closed convex set. We claim that $K \neq \emptyset$. Indeed, by \cite{Ma25}, we know that $N \subset M$ has the weak Dixmier property and more precisely that $E_{N' \cap M}$ lies in the closed convex hull of $\{ \Ad(u) \mid u \in \cU(N) \} \subset \UCP(M)$. Hence there exists a ucp map $\chi \in \UCP( \B(\rL^2(M))$ in the closed convex hull of $\{ \Ad(\lambda(u)) \mid u \in \cU(N) \} \subset \UCP(M)$ such that 
\begin{itemize}
\item $\chi(e_N)=e_N$
\item $\chi \circ \lambda = \lambda \circ E_{N' \cap M}$
\item $\chi \circ \rho = \rho$
\end{itemize} 
 Define a state $\Phi_0 \in \B(\rL^2(M))^*$ by $\Phi_0(T)=\langle \chi(T) \varphi^{1/2},\varphi^{1/2} \rangle$ for all $T \in \B(\rL^2(M))$. Then $\Phi_0 \in K$.

Now, consider the one-parameter automorphism group $$\sigma : \R \ni t \mapsto \Ad(\Delta_\varphi^{\ri t}) \in \Aut(\B(\rL^2(M)).$$
Clearly, $K$ is globally $\sigma$-invariant in the sense that $\Phi \circ \sigma_t \in K$ for every $\Phi \in K$ and $t \in \R$. The same proof as in \cite[Lemma 4.1]{Ma18} shows that $K$ is actually strongly $\sigma$-invariant : we have $\Phi \circ \sigma_\mu \in K$ for every $\Phi \in K$ and every probability measure $\mu \in \mathrm{Prob}(\R)$. From \cite[Proposition 3.10]{Ma18}, we conclude that $K$ contains some strongly $\sigma$-invariant state $\Phi_1$. 

Then thanks to Lemma \ref{centralizing}, we can find $\Phi_2 \in \B(\rL^2(M))^*$ such that 
 \begin{itemize}
 \item $\Phi_2(e_N)=\Phi_1(e_N)=1$.
 \item $\Phi_2(f(\log\Delta_\varphi))=f(0)$ for every $f \in C_0(\R)$.
  \item $\Phi_2(\lambda(a)\rho(b))=\Phi_1(\lambda(a)\rho(b))=\langle E_{N' \cap M}(a) \varphi^{1/2} b, \varphi^{1/2} \rangle$ for all $a,b \in \rB(N \subset M, \varphi)^{\beta^\varphi}$.
\end{itemize}
Now take $\pi$ as in Lemma \ref{ucp maps} and let $\Phi_3=\Phi_2 \circ \pi$, then we have
 \begin{itemize}
 \item $\Phi_3(e_N)=\Phi_2(e_N)=1$.
 \item $\Phi_3 \circ \lambda = \Phi_2 \circ \lambda \circ E_{\rB(N \subset M,\varphi)^{\beta^\varphi}} = \varphi$.
 \item $\Phi_3 \circ \rho = \Phi_2 \circ \rho \circ E_{\rB(N \subset M,\varphi)^{\beta^\varphi}} = \varphi$.
 \item $\Phi_3(f(\log\Delta_\varphi))=\Phi_2(f(\log\Delta_\varphi))=f(0)$ for every $f \in C_0(\R)$.
  \item $\Phi_3(\lambda(a)\rho(b))=\Phi_2(\lambda(a)\rho(b))=\langle E_{N' \cap M}(a) \varphi^{1/2} b, \varphi^{1/2} \rangle$ for all $a,b \in \rB(N \subset M, \varphi)^{\beta^\varphi}$.
\end{itemize}
The last item follows from the fact that $\lambda(a)$ and $\rho(b)$ are  in the multiplicative domain of $\pi$ for all $a,b \in \rB(N \subset M, \varphi)^{\beta^\varphi}$.

Now, we can apply Lemma \ref{binormal state bicentralizer} to deduce that $\Phi_3(\lambda(a) \rho(b))= \langle a \varphi^{1/2} b, \varphi^{1/2} \rangle$ for all $a,b \in \rB(N \subset M,\varphi)$. Therefore, by the last item above, we conclude that  $$\langle a \varphi^{1/2} b, \varphi^{1/2} \rangle=\langle E_{N' \cap M}(a) \varphi^{1/2} b, \varphi^{1/2} \rangle$$ for all $a,b \in \rB(N \subset M,\varphi)^{\beta^\varphi}$, and this means that $\rB(N \subset M,\varphi)^{\beta^\varphi}=N' \cap M$.
\end{proof}

Theorem \ref{letter thm amenable general} follows from Theorem \ref{first part} and \cite[Theorem C]{AHHM18} in the case where $N$ is a factor of type $\III_1$. For the general case, we then proceed as in \cite[Theorem 10.20]{Ma25}. Theorem \ref{letter kadison}.(2) is deduced from Theorem \ref{letter thm amenable general} and then Theorem \ref{letter kadison}.(1) applied to $P$.

\section{Eigenvectors of the relative bicentralizer flow}
The second part of Theorem \ref{main theorem} is actually a consequence of the first part applied to a crossed product inclusion.

\begin{thm} \label{second part}
Let $N \subset M$ be an inclusion of von Neumann algebras with expectation, where $N$ is of type $\III_1$. Let $\varphi$ be a faithful normal state on $N$. Then for every $t \in \R$ and every $x \in M$, the following are equivalent :
\begin{enumerate}
\item $x \in \rB(N \subset M,\varphi)$ and $\beta_\lambda^\varphi (x)=\lambda^{\ri t} x$ for all $ \lambda \in \R^*_+$.
\item $x \in \rb(N \subset M,\varphi)$ and $\beta_\lambda^\varphi (x)=\lambda^{\ri t} x$ for all $ \lambda \in \R^*_+$.
\item $xy=\sigma_t^\varphi(y)x$ for all  $y \in N$.
\end{enumerate}
\end{thm}
\begin{proof}

 (1) $\Rightarrow$ (3). Suppose that there exists $t \in \R$ and $x \in \rB(N \subset M,\varphi)$ such that $\beta_\lambda^\varphi(x)=\lambda^{\ri t} x$ for every $\lambda \in \R^*_+$. Let $\widetilde{M}=M \rtimes_{\sigma_t^\varphi} \Z$ be the crossed product which is generated by a copy of $M$ and a unitary $u$ such that $uau^*=\sigma_t^\varphi(a)$ for all $a \in M$. Observe that $N \subset \widetilde{M}$ is with expectation. Moreover, we have $u \in \rB(N \subset \widetilde{M},\varphi)$ and  $\beta_\lambda^\varphi(u)=\lambda^{\ri t} u$ for all $\lambda \in \R^*_+$. This means that $u^*x \in \rB(N \subset \widetilde{M},\varphi)^{\beta^\varphi}$. By Theorem \ref{first part}, we conclude that $u^*x \in N' \cap \widetilde{M}$ which means that $$xy=uu^*xy=uyu^*x=\sigma_t^\varphi(y)x$$
 for all $y \in N$. 
 
 (3) $\Rightarrow$ (2). Take $x \in M$ and $t \in \R$ such that $xy=\sigma_t^\varphi(y)x$ for all $y \in N$. Then $\varphi^{-\ri t}x \in N' \cap c(M)$. Thus $x \in \varphi^{\ri t} \cdot (N' \cap c(M))$, hence $x \in \rb(N \subset M,\varphi)$ by definition of $\rb(N \subset M,\varphi)$. Moreover, since $\varphi^{-\ri t}x \in N' \cap c(M)$, we have $\beta_\lambda^\varphi(\varphi^{-\ri t}x)=\varphi^{-\ri t}x$ for every $\lambda \in \R^*_+$. Since $\beta_\lambda^\varphi(\varphi^{\ri t})=\lambda^{\ri t} \varphi^{\ri t}$, we conclude that $\beta_\lambda^\varphi(x)=\lambda^{\ri t} x$.
 
 (2) $\Rightarrow$ (1). This follows from \cite[Proposition 8.4]{Ma25}.
\end{proof}

\begin{df}
Let $N \subset M$ be an inclusion of von Neumann algebras. Define its relative $T$-invariant by
$$ T(N \subset M):=\{ t \in \R \mid \text{there exists a nonzero } x \in M \text{ such that } xy=\sigma_t^\varphi(y)x \text{ for all } y \in N\}$$
where $\varphi$ is any faithful normal semifinite weight on $N$.
\end{df}
\begin{rem}
The invariant $T(N \subset M)$ is independent of the choice of $\varphi$ because of Connes's Radon-Nikodym cocycle theorem : if $\psi$ is another faithful normal semifinite weight on $N$, then $\sigma_t^\psi = \Ad(u) \circ \sigma_t^\varphi$ for some unitary $u \in \cU(N)$.

If $N$ is an irreducible subfactor of $M$, then $T(N \subset M)$ is a subgroup of $\R$.
\end{rem}

\begin{prop}
Let $N \subset M$ be an inclusion of von Neumann algebras. Suppose that $N$ is a type $\III_1$ factor, that $N \subset M$ is with expectation and that $M$ has separable predual. Then $T(N \subset M)$ is at most countable.
\end{prop}
\begin{proof}
Take $E : M \rightarrow N$ a faithful normal condition expectation. Take $t_1 \neq t_2$ in $\R$ and two nonzero $x_i \in M, i=1,2$ such that $x_iy=\sigma_{t_i}^\varphi(y)x_i$ for all $y \in N$. Then $x_1^*x_2y=\sigma_{t_2-t_1}^\varphi(y)x_1^*x_2$ for all $y \in N$. Applying $E$ to this equation, we obtain $$E(x_1^*x_2)y=\sigma_{t_2-t_1}^\varphi(y)E(x_1^*x_2)$$ for all $y \in N$. But since $N$ is of type $\III_1$, the modular automorphism $\sigma_{t_2-t_1}^\varphi$ is outer. This means that $E(x_1^*x_2)=0$. Since $M$ has separable predual, an $E$-orthogonal family in $M$ is at most countable. We conclude that $T(N \subset M)$ is at most countable.
\end{proof}

We now have the following corollary of Theorem \ref{second part}.

\begin{cor} \label{criterion fixed point}
Let $N \subset M$ be an inclusion of von Neumann algebras with expectation, where $N$ is of type $\III_1$. Let $\varphi$ be a faithful normal state on $N$. Then for every $\lambda  \in \R^*_+ \setminus \{1\}$ we have
$$ \rB(N \subset M,\varphi)^{\beta^\varphi_\lambda}=\rb(N \subset M,\varphi)^{\beta^\varphi_\lambda}$$
and this algebra is equal to $N' \cap M$ is and only if there is no nonzero $t \in T(N \subset M)$ such that $\lambda^{\ri t}=1$. 
\end{cor}
\begin{proof}
The restriction of $\beta_\lambda^\varphi$ to $\rB(N \subset M,\varphi)^{\beta^\varphi_\lambda}$ is periodic, hence $\rB(N \subset M,\varphi)^{\beta^\varphi_\lambda}$ and $\rb(N \subset M,\varphi)^{\beta^\varphi_\lambda}$ are generated by eigenvectors of $\beta^\varphi$. Moreover, if $x \in \rB(N \subset M,\varphi)^{\beta_\lambda^\varphi}$ is a nonzero eigenvector of $\beta^\varphi$, then there exists $t \in \R$ such that $\beta_\mu^\varphi(x)=\mu^{\ri t} x$ for all $\mu \in \R^*_+$. Since $x$ is fixed by $\beta_\lambda^\varphi$, we must have $\lambda^{\ri t}=1$.

The conclusion follows from Theorem \ref{second part}.
\end{proof}

Corollary \ref{letter cor amenable III1} follows from \cite[Theorem C]{AHHM18} and Corollary \ref{criterion fixed point}.

 \section{Eigenvectors of the modular flow in the relative bicentralizer}
 
 \begin{thm} \label{eigenvectors modular flow}
Let $N \subset M$ be an inclusion of von Neumann algebras with a faithful normal conditional expectation $E_N : M \rightarrow N$. Take a faithful state $\varphi \in M_*$ such that $\varphi=\varphi \circ E_N$.

Then every eigenvector of $\sigma^\varphi|_{\rB(N \subset M,\varphi)}$ is contained in $\rb(N \subset M,\varphi)$. Equivalently, we have $$\rB(N \subset M,\varphi)^{\sigma_t^\varphi}=\rb(N \subset M,\varphi)^{\sigma_t^\varphi}$$
for every $t \in \R \setminus \{0\}$.
\end{thm}
\begin{proof}
Fix $t \in \R \setminus \{0\}$. Let $\widetilde{N} = N \rtimes_{\sigma_t^\varphi} \Z$ and $\widetilde{M}=M \rtimes_{\sigma_t^\varphi} \Z$. Let $\widetilde{\varphi}$ be the natural extension of $\varphi$ to $\widetilde{M}$. 

Our goal is to show that \begin{equation} \label{key inclusion}
\rB(N \subset M,\varphi)^{\sigma_t^\varphi} \subset \rB(\widetilde{N} \subset \widetilde{M},\widetilde{\varphi}).
\end{equation}
The theorem follows easily from this inclusion. Indeed,since $\widetilde{N}$ has no type $\III_1$ summand, by \cite[Theorem E.(2)]{Ma25}, it satisfies the relative bicentralizer conjecture 
$$ \rB(\widetilde{N} \subset \widetilde{M},\widetilde{\varphi}) = \rb(\widetilde{N} \subset \widetilde{M},\widetilde{\varphi}).$$
Thus, if the inclusion (\ref{key inclusion}) holds, we can use \cite[Proposition 6.10.3]{Ma25} to conclude that
$$ \rB(N \subset M,\varphi)^{\sigma_t^\varphi}  \subset  \rb(\widetilde{N} \subset \widetilde{M},\widetilde{\varphi}) \cap M \subset \rb( N \subset \widetilde{M}, \varphi) \cap M = \rb(N \subset M, \varphi).$$

Now, let us prove the inclusion (\ref{key inclusion}). 

Let $\widetilde{\lambda}, \widetilde{\rho} : \widetilde{M} \rightarrow \B(\rL^2(\widetilde{M}))$ denote the left and right action of $\widetilde{M}$ on its standard form. Let $e_{\widetilde{N}} \in \B(\rL^2(\widetilde{M}))$ denote the Jones projection associated to the $\widetilde{\varphi}$-preserving conditional expectation $E_{\widetilde{N}} : \widetilde{M} \rightarrow \widetilde{N}$. By \cite[Lemma 5.10 and Proposition 7.15.(2)]{Ma25} there exists a state $\Phi_0 \in \B(\rL^2(\widetilde{M}))^*$ such that 
 \begin{itemize}

 \item $\Phi_0(e_{\widetilde{N}})=1$.
 \item $\Phi_0(f(\log\Delta_{\widetilde{\varphi}}))=f(0)$ for every $f \in C_0(\R)$.
  \item $\Phi_0(\widetilde{\lambda}(a)\widetilde{\rho}(b))=\langle E_{\rB(\widetilde{N} \subset \widetilde{M},\widetilde{\varphi})}(a) \widetilde{\varphi}^{1/2} b, \widetilde{\varphi}^{1/2} \rangle$ for all $a,b \in \widetilde{M}$.
\end{itemize}
Now, since $\widetilde{M} = M \rtimes_{\sigma_t^\varphi} \Z$, we can make an identification $$\rL^2(\widetilde{M}) = \rL^2(M) \otimes \ell^2(\Z)$$
in such a way that
\begin{itemize}
\item $\widetilde{\lambda}(a)=\lambda(a) \otimes 1$ for every $a \in M$.
\item $\widetilde{\rho}(a)=\sum_{n \in \Z} \rho(\sigma_{nt}^\varphi(a)) \otimes p_n$ for every $a \in M$, where $p_n$ are the canonical rank one projections on $\ell^2(\Z)$. In particular, $\widetilde{\rho}(a)=\rho(a) \otimes 1$ when $a \in M^{\sigma_t^\varphi}$.
\item  $e_{\widetilde{N}} = e_N \otimes 1$
\item $\Delta_{\widetilde{\varphi}} = \Delta_\varphi \otimes 1$. 
\end{itemize}
Therefore, if we define a state $\Phi_1 \in \B(\rL^2(M))^*$ by letting $\Phi_1(T)=\Phi_0(T \otimes 1)$, then we get 
 \begin{itemize}

 \item $\Phi_1(e_{N})=\Phi_0(e_{\widetilde{N}})=1$.
 \item $\Phi_1(f(\log\Delta_{\varphi}))=\Phi_0(f(\log\Delta_{\widetilde{\varphi}}))=f(0)$ for every $f \in C_0(\R)$.
  \item $\Phi_1(\lambda(a) \rho(b))=\Phi_0(\widetilde{\lambda}(a) \widetilde{\rho}(b))=\langle E_{\rB(\widetilde{N} \subset \widetilde{M},\widetilde{\varphi})}(a) \widetilde{\varphi}^{1/2} b, \widetilde{\varphi}^{1/2} \rangle$ for all $a,b \in M^{\sigma_t^\varphi}$.
\end{itemize}
Now, since $\Z$ is amenable, we can find $\theta \in \UCP(\B(\rL^2(M))$ in the closed convex hull of $\{ \Ad(\Delta_\varphi^{\ri t})^n \mid n \in \Z \}$ such that 
\begin{itemize}
\item $\theta \circ \lambda = \lambda \circ E_{M^{\sigma_t^\varphi}}$.
\item $\theta \circ \rho = \rho \circ E_{M^{\sigma_t^\varphi}}$.
\item $\theta(e_N)=e_N$
\item $\theta(f(\log\Delta_{\varphi}))=f(\log\Delta_{\varphi})$ for every $f \in C_0(\R)$.
\end{itemize}
Let $\Phi_2=\Phi_1 \circ \theta$. Then we have
\begin{itemize}
 \item $\Phi_2(e_{N})=\Phi_1(e_{N})=1$.
 \item $\Phi_2(f(\log\Delta_{\varphi}))=\Phi_1(f(\log\Delta_{\varphi}))=f(0)$ for every $f \in C_0(\R)$.
\item $\Phi_2(\lambda(a) \rho(b))=\Phi_1(\lambda(a) \rho(b))=\langle E_{\rB(\widetilde{N} \subset \widetilde{M},\widetilde{\varphi})}(a) \widetilde{\varphi}^{1/2} b, \widetilde{\varphi}^{1/2} \rangle$ for all $a,b \in M^{\sigma_t^\varphi}$.
\item $\Phi_2 \circ \lambda = \Phi_1 \circ \lambda \circ E_{M^{\sigma_t^\varphi}} = \varphi$
\item $\Phi_2 \circ \rho = \Phi_1 \circ \rho \circ E_{M^{\sigma_t^\varphi}} = \varphi$
 \end{itemize}
By applying Lemma \ref{binormal state bicentralizer}, we conclude that
$$\langle E_{\rB(\widetilde{N} \subset \widetilde{M},\widetilde{\varphi})}(a) \widetilde{\varphi}^{1/2} b, \widetilde{\varphi}^{1/2} \rangle =\langle a \varphi^{1/2} b , \varphi^{1/2} \rangle$$ for all $a,b \in \rB(N \subset M, \varphi)^{\sigma_t^\varphi}$. This means that 
$$ \rB(N \subset M, \varphi)^{\sigma_t^\varphi} \subset \rB(\widetilde{N} \subset \widetilde{M},\widetilde{\varphi})$$
as we wanted.
\end{proof}

\begin{prop}
Let $N \subset M$ be an irreducible subfactor with expectation where $N$ is of type $\III_1$. Take a faithful state $\varphi \in M_*$ such that $N$ is globally invariant under $\sigma^\varphi$. Then $c(N)' \cap c(M)=\C$ if and only if $\rb(N \subset M,\varphi)^{\sigma^\varphi} =\C$.
\end{prop}
\begin{proof}
Take $\psi$ a dominant weight on $N$ and extend it to $M$ by using the unique faithful normal conditional expectation onto $N$. Then, by \cite[Theorem 6.20]{Ma25} and \cite[Theorem 6.13]{Ma25}, we have an isomorphism  
$$\rb^{\psi,\varphi} :  \rb(N \subset M,\varphi) \rightarrow \rb(N \subset M,\psi)=N_\psi' \cap M$$ 
and this isomorphism is equivariant with respect to the modular flow $\sigma^\varphi$ and $\sigma^\psi$. Thus $\rb^{\psi,\varphi}$ sends $\rb(N \subset M,\varphi)^{\sigma^\varphi}$ onto $(N_\psi' \cap M)^{\sigma^\psi}=N_\psi' \cap M_\psi'$.

The conclusion follows from the fact that the inclusion $c(N) \subset c(M)$ is isomorphic to $N_\psi \subset M_\psi$ by Takesaki's duality.
\end{proof}

Theorem \ref{ergodic flow inclusion} is a direct consequence of Theorem \ref{eigenvectors modular flow} and Theorem \cite[Theorem 4.3]{MV23}.

\bibliographystyle{plain}

\end{document}